\documentstyle[amsfonts,amssymb,amsthm,amsmath, xcolor, 12pt]{article}

\title{On spaces with a $\pi$-base whose elements have an H-closed closure}
\author{Davide Giacopello}
\date{}

\newtheorem{theorem}{Theorem}[section]
\newtheorem{corollary}[theorem]{Corollary}
\newtheorem{question}[theorem]{Question}
\newtheorem{example}[theorem]{Example}
\newtheorem{lemma}[theorem]{Lemma}

\begin{document}
\maketitle
\begin{abstract}
	We deal with the class of Hausdorff spaces having a $\pi$-base whose elements have an H-closed closure. Carlson proved that $|X|\leq 2^{wL(X)\psi_c(X)t(X)}$ for every quasiregular space $X$ with a $\pi$-base whose elements have an H-closed closure. We provide an example of a space $X$ having a $\pi$-base whose elements have an H-closed closure which is not quasiregular (neither Urysohn) such that $|X|> 2^{wL(X)\chi(X)}$ (then $|X|> 2^{wL(X)\psi_c(X)t(X)}$). Still in the class of spaces with a $\pi$-base whose elements have an H-closed closure, we establish the bound $|X|\leq2^{wL(X)k(X)}$ for Urysohn spaces and we give an example of an Urysohn space $Z$ such that $k(Z)<\chi(Z)$. Lastly, we present some equivalent conditions to the Martin's Axiom involving spaces with a $\pi$-base whose elements have an H-closed closure and, additionally, we prove that if a quasiregular space has a $\pi$-base whose elements have an H-closed closure then such space is Baire.
\end{abstract}

 
\bigskip 


\section{Introduction}
Throughout the paper we mean by \lq\lq space\rq\rq, an Hausdorff topological space. And for semplicity if a space has a $\pi$-base whose elements have an H-closed (compact) closure, we say that the space has an H-closed (compact) $\pi$-base.
In 1978, Bell, Ginsburg, and Woods established, in \cite{BGW}, that if a space $X$ is normal, then $|X|\leq 2^{wL(X)\chi(X)}$, and they constructed an example of a Hausdorff space in which the aforementioned inequality does not hold. In the same paper, they posed the question of whether the hypothesis of normality could be weakened in some way. In \cite{DP}, Dow and Porter achieved a powerful result: $|X|\leq 2^{\psi_c(X)}$ for every H-closed space $X$. 
In \cite{BC}, Bella and Carlson proved that the inequality $|X|\leq 2^{wL(X)t(X)\psi(X)}$ is true for every regular space $X$ with a compact $\pi$-base.
In \cite{BCG1}, Bella, Carlson and Gotchev proved that almost the same inequality (with the closed pseudocharacter $\psi_c(X)$ instead of the pseudocharacter $\psi(X)$) holds also for spaces with a compact $\pi$-base. That is,

\begin{theorem}\rm\cite{BCG1}\label{BCG1}
	If $X$ is a space with a compact $\pi$-base. Then $|X|\leq 2^{wL(X)t(X)\psi_c(X)}$.
\end{theorem}

Since H-closedness is a natural generalization of compactness, Bella, Carlson and Gotchev posed the following question. 

\begin{question}\rm \cite{BCG1}\label{BCG} 
	Let $X$ be a space with a $\pi$-base whose elements have H-closed closures. Is it true that $|X|\leq 2^{wL(X)t(X)\psi_c(X)}$?
\end{question}

In \cite{BC}, some further investigations on spaces having a $\pi$-base with some properties on the closure of the elements have led to the following result.

\begin{theorem}\rm\cite{BC}\label{BC}
	Let $X$ be a space and ${\cal B}$ an open $\pi$-base. Suppose for all $B \in {\cal B}$ that $\overline{B}$ is $H$-closed, (or normal, or Lindelöf, or has the ccc). Then $d_\theta(X) \leq 2^{w L(X) \chi(X)}$ and if $X$ is quasiregular or Urysohn then $|X| \leq 2^{w L(X) \chi(X)}$.
\end{theorem}

Then we pose the following questions.

\begin{question}\rm \label{quasi}
	Does the inequality $|X| \leq 2^{wL(X) t(X) \psi_c(X)}$ hold for quasiregular spaces having an H-closed $\pi$-base? 
\end{question}
\begin{question}\rm \label{ury}
	Does the inequality $|X| \leq 2^{wL(X) t(X) \psi_c(X)}$ hold for Urysohn spaces having an H-closed $\pi$-base? 
\end{question}

Recall that the Banach-Mazur game on the space $X$ is played by two players {\sc Alice} and {\sc Bob} in $\omega$-many innings. At the beginning of the game, {\sc Alice} chooses a nonempty open set $U_0$ and {\sc Bob} responds by choosing a nonempty open set $V_0\subset U_0$. At the $n$-th inning ($n>0$), {\sc Alice} chooses a nonempty $U_n\subset V_{n-1}$ and {\sc Bob} responds by choosing a nonempty open set $V_n\subset U_n$, and so on. The player {\sc Bob} wins if and only if $\bigcap _{n\in\omega}V_n\neq\emptyset$.
Banach, Mazur and Oxtoby proved that the spave $X$ has the Baire property if and only if {\sc Alice} does not have a winning strategy in the Banach-Mazur game on $X$. 
A space $X$ is said Choquet if {\sc Bob} has a winning strategy in the Banach-Mazur game on $X$. Choquet spaces were introduced in 1975 by White who called them weakly $\alpha$-favorable spaces. Choquet spaces are Baire. A Bernstein subset of reals witnesses that Baire spaces need not be Choquet.\\
Bella, Carlson and Gotchev proved the following result.

\begin{theorem}\rm\cite{BCG2}
A space $X$ with a $\pi$-base whose elements have closures that are compact is Choquet.
\end{theorem}

It is natural to ask the following question.

\begin{question}\rm\label{alpha}
Is any space $X$ with an H-closed $\pi$-base a Choquet space?
\end{question}

In Section \ref{1} we give a negative answer to Question \ref{BCG} (see Example \ref{exe}) providing an example of a space $X$ having an H-closed $\pi$-base which is not quasiregular (neither Urysohn) such that $|X|> 2^{wL(X)\chi(X)}$ (then $|X|> 2^{wL(X)\psi_c(X)t(X)}$). Moreover Carlson answered Question \ref{quasi} (see Theorem \ref{HclosedQR}) proving that $|X|\leq 2^{wL(X)\psi_c(X)t(X)}$ for every quasiregular space with an H-closed $\pi$-base. Example \ref{exe} shows, inter alia, that the quasiregular hypotesis is essential. In the following, we present some further results on spaces with a $\pi$-base whose elements have compact boundaries.\\ 
Additionally, we establish the bound $|X|\leq2^{wL(X)k(X)}$ for Urysohn spaces having an H-closed $\pi$-base. The function $k(X)$, defined later on, was introduced by Alas and Kocinac in \cite{AK}. We provide an example of an Urysohn space $Z$ having an H-closed $\pi$-base such that $k(Z)<\chi(Z)$.\\
In Section \ref{2} we present some equivalent conditions to the Martin's Axiom involving spaces with an H-closed $\pi$-base and, additionally, we give a partial answer to Question \ref{alpha} (see Theorem \ref{Baire}).

\section{Notation and terminology}

For a subset $A$ of a topological space $X$ we will denote by $[A]^{<\lambda}$  $([A]^\lambda)$ the family of all subsets of $A$ of cardinality $<\lambda$ $(= \lambda)$.
We consider cardinal invariants of topological spaces (for further informations see \cite{E, Ho, J}). 
In particular, given a topological space $X$, we will denote with $d(X)$ its density, $\chi(X)$ its character, $\pi \chi(X)$ its $\pi$-character, $t(X)$ its tightness, $\psi(X)$ its pseudocharacter, $\psi_c(X)$ its closed pseudocharacter, $L(X)$ its Lindel\"of number.\\
Given a subset $A\subseteq X$ the $\theta$-closure (see \cite{V}) of $A$, $cl_\theta(A)$, is the subset $\{x\in X: \overline{U}\cap A\not=\emptyset$ for every open neighborhood $U$ of $x \}$. Clearly, $\overline{A}\subseteq cl_\theta(A)$ for every $A\subseteq X$.
Recall that, given a space $X$, the cardinal invariant $d_\theta (X)$, closely related to the density $d(X)$, is defined as follows. A subspace
$D\subseteq X$ is $\theta$-dense in $X$ if $D\cap\overline U\neq\emptyset$ for every non-empty open set $U$ of $X$. The $\theta$-density $d_\theta(X)$ is the least cardinality
of a $\theta$-dense subspace of $X$. Observe that $d_\theta(X)\leq d(X)$ for any space $X$
(see \cite{BCMP} and \cite{C} for more details about $\theta$-density). The $\theta$-tightness, $t_\theta(X)$, is the least cardinal $\kappa$ such that if $x\in cl_\theta(A)$, then there exists $B\subseteq A$ such that $|B|\leq \kappa$ and $x\in cl_\theta(B)$. If $X$ is Urysohn is possible to define the $\theta$-pseudocharacter, $\psi_\theta(X)$, that is the least cardinal $\kappa$ such that for each point $x\in X$ there exists a family ${\cal V}_x$ of open subsets containing $x$ such that $|{\cal V}_x|\leq \kappa $ and $\{x\}=\bigcap_{V\in{\cal V}_x} cl_\theta (\overline{V})$. Clearly, $\psi_c(X)\leq \psi_\theta(X)$ for every Urysohn space $X$.\\
Given a space $X$, its weak Lindel\"of number, $wL(X)$, is the least cardinal $\kappa$ such that for every open cover $\cal U$ of $X$ there exists a subcover ${\cal V}\in [{\cal U}]^{\leq\kappa}$ such that $\overline{\bigcup_{V\in{\cal V}}V}=X$.\\ 
A space is called $H$-closed if it is Hausdorff and it is closed in every Hausdorff space in which it is embedded. Moreover, a more operative characterization of H-closedness of a space $X$ is given by the following statement: \lq\lq for every open cover $\cal U$ of $X$ there exists a finite subfamily ${\cal V}\in [{\cal U}]^{<\omega}$ such that $\overline{\bigcup {\cal V}}=X$\rq\rq. Two other characterizations of H-closed spaces are that each open filter (open ultrafilter) on the space has a nonempty adherence (convergers, respectiely).\\
A subset $U\subseteq X$ is called regular closed if $U=\overline{int(U)}$ and regular open if $U=int(\overline{U})$. It is known that if a space is H-closed every regular closed subset of it inherits the property. Given a space $(X,\tau)$, the semiregularization of $X$, denoted by $X_s$, is the space $X$ endowed with the topology generated by the basis $\{int(\overline{U}): U\in\tau\}$ of all the regular open sets of $X$.\\
Let $({\Bbb P}, \leq, 1)$ be a poset. $D\subseteq {\Bbb P}$ is a dense subset (in the sense of posets) if for every $p\in {\Bbb P}$ there exists $d\in D$ such that $d\leq p$. $p,q\in {\Bbb P}$ are called compatible ($p\not\perp q$) if there exists $r\in {\Bbb P}$ such that $r\leq p$ and $r\leq q$, otherwise they are called incompatible ($p\perp q$). A poset $({\Bbb P}, \leq, 1)$ is said to have the countable chain condition (ccc is the sense of posets) if every antichain (subset of ${\Bbb P}$ having pairwise incompatible elements) is countable. A subset $G\subseteq \Bbb P$ is called filter if (a) the maximum $1$ belongs to $G$; (b) for every $p,q\in G$ there exists $r\in G$ such that $r\leq p$ and $r\leq q$; (c) for each $p,q\in \Bbb P$, if $p\in G$ and $p\leq q$, then $q\in G$.

\section{Upper bounds on the cardinality of spaces with an H-closed $\pi$-base}\label{1}
It is known that for regular spaces the H-closedness and the compactness are equivalent properties. In the following we provide an example of a space having the property we are dealing with and does not have a compact $\pi$-base. 
\begin{example}\rm
	An axample of a non quasiregular (hence non regular) space having an H-closed $\pi$-base, but it does not have a compact $\pi$-base.
\end{example}
Consider the space $(\Bbb R,\tau_{\Bbb Q})$ where $\tau_{\Bbb Q}$ is the topology generated by the open set of the form $\{x\}\cup ( U\cap {\Bbb Q}$), with $x\in {\Bbb R}\setminus \Bbb Q$ and $U$ an open set in the standard topology on $\Bbb R$. Let $X=[0,2]$ with the topology inherited from $(\Bbb R,\tau_{\Bbb Q})$. $X$ is a Hausdorff non regular nowhere locally compact space. Then it does not have a compact $\pi$-base. In \cite{H} Herrlich  proved that $X$ is H-closed then it admits a $\pi$-base whose elements have H-closed closures. Now we prove that there is not a $\pi$-base whose elements have quasiregular closure. Suppose that ${\cal B}$ is a $\pi$-base of $X$ having elements with quasiregular closures. Let's take the open set $(0,2)\cap {\Bbb Q}$, then there exists $U\in{\cal B}$ such that $U\subseteq (0,2)\cap{\Bbb Q}$, by hypotesis $\overline{U}$ is quasiregular. Therefore there is $V\in {\cal B}$ such that $V\subseteq\overline{V}\subseteq U$ so $\overline{V}\subseteq {\Bbb Q}$, which is a contradiction. By Corollary \ref{equi} (see below) follows that $X$ is not quasiregular. $\hfill\triangle$

\bigskip

In \cite{BC} it is proved that  $|X|\leq 2^{wL(X)\chi(X)}$ holds for quasiregular (or Urysohn) spaces having an H-closed $\pi$-base. Later on we show that such bound can be a little improved. 

\begin{lemma}\rm\label{qr}(Carlson)
	Let $X$ be a space with a $\pi$-base ${\cal B}$ whose elements have quasiregular closure. Then $X$ is quasiregular.
\end{lemma}
\begin{proof}
	Let $U$ be a non-empty open subset of $X$. There exists $V\in{\cal B}$ such that $V\subseteq U$ and $\overline{V}$ is quasiregular. The set $V=V\cap\overline{V}$ is open in $\overline{V}$. As $\overline{V}$ is quasiregular, there exists a non-empty open subset $W$ of $X$ such that $\emptyset\neq W\cap\overline{V}\subseteq \overline{(W\cap \overline{V})}^{\overline{V}}=\overline{(W\cap\overline{V})}\cap\overline{V}\subseteq V$.
	
	As $W\cap\overline{V}\not=\emptyset$ then $W\cap V\not=\emptyset$. Furthermore, we have $\overline{(W\cap V)}\subseteq \overline{(W\cap\overline{V})}$ and $\overline{(W\cap V)}\subseteq\overline{V}$. Thus $\overline{(W\cap V)}\subseteq V\subseteq U$. As $W\cap V$ is open in $X$ and non-empty, we conclude $X$ is quasiregular. 
\end{proof}
\begin{corollary}\rm\label{equi}
	A space $X$ has a $\pi$-base ${\cal B}$ whose elements have quasiregular closure if, and only if, $X$ is quasiregular.
\end{corollary}
The previous corollary gives us the motivation to place the hypotesis of quasiregularity on the space instead of putting it on the closure of the elements of a $\pi$-base.\\
Recall that, in \cite{BCam}, Bella and Cammaroto proved that $|X|\leq d(X)^{t(X)\psi_c(X)}$ for every Hausdorff space $X$. 

\begin{theorem}\rm\label{HclosedQR}(Carlson)
	Let $X$ be a quasiregular space with an H-closed $\pi$-base. Then $|X|\leq 2^{wL(X)t(X)\psi_c(X)}$.
\end{theorem}

\begin{proof}
	Let $\kappa=wL(X)t(X)\psi_c(X)$ and let ${\cal B}$ be a $\pi$-base of non-empty open sets with closures that are H-closed and quasiregular. For each $B\in{\cal B}$, as $\overline{B}$ is H-closed, by the Dow-Porter's result (in \cite{DP}) we have that $|\overline{B}|\leq 2^{\psi_c(\overline{B})}\leq 2^{\psi_c(X)}\leq 2^\kappa$.
	Since $\psi_c(X)\le \kappa$, for each $x\in X$ we can fix a collection ${\cal V}_x$ 
	of open neighborhoods of $x$ such that $|{\cal V}_x|\le \kappa$ and $\bigcap\{\overline{V}:V\in {\cal V}_x\}=\{x\}$. Without 
	loss of generality we may assume that each ${\cal V}_x$ is closed under finite intersections.
	
	We will construct by transfinite recursion a non-decreasing chain of open sets $\{U_\alpha : \alpha < \kappa^+\}$ such that 
	\begin{itemize}
		\item[(1)] $|\overline{U_\alpha}| \leq 2^\kappa$ for every $\alpha < \kappa^+$, and 
		\item[(2)] if $X\setminus\overline{\bigcup {\cal M}} \not= \emptyset$ for some 
		${\cal M} \in [\bigcup\{{\cal V}_x : x \in \overline{U_\alpha}\}]^{\leq \kappa}$, 
		then there is $B_{\cal M}\in{\cal B}$ such that $B_{\cal M}\subset U_{\alpha+1}\setminus\overline{\bigcup {\cal M}}$.
	\end{itemize}
	
	Let $B_0\in{\cal B}$ be arbitrary. We set $U_0=B_0$.  Then $|\overline{U_0}|\leq 2^\kappa$. 
	If $\beta =\alpha + 1$, for some $\alpha$, then for every 
	${\cal M} \in [\bigcup\{{\cal V}_x : x \in \overline{U_\alpha}\}]^{\leq \kappa}$ such that $X\setminus\overline{\bigcup{\cal M}} \not= \emptyset$, 
	we choose $B_{{\cal M}} \in {\cal B}$ such that $B_{{\cal M}} \subseteq X\setminus\overline{\bigcup {\cal M}}$. 
	We define $U_{\beta} = U_\alpha\cup\bigcup\{B_{\cal M} : {\cal M} \in [\bigcup\{{\cal V}_x : x \in \overline{U_\alpha}\}]^{\leq \kappa}, X\setminus\overline{\bigcup{\cal M}}\not=\emptyset\}$. Therefore, by Bella-Cammaroto's inequality, we have that 
	$|\overline{U_{\beta}}| \leq 2^\kappa$. If $\beta<\kappa^+$ is a limit ordinal we let 
	$U_\beta = \bigcup_{\alpha<\beta} U_\alpha$. Then clearly $|U_\beta|\leq 2^\kappa$, hence 
	$|\overline{U_\beta}| \leq 2^\kappa$.
	
	Let $F = \bigcup\{\overline{U_\alpha} : \alpha < \kappa^+\}$. Then $|F| \le 2^\kappa$. 
	Since $t(X)\leq\kappa$, $F$ is closed and therefore $F = \overline{\bigcup\{U_\alpha : \alpha < \kappa^+\}}$. 
	Thus, $F$ is a regular-closed set. 
	
	We will show that $X = F$. Suppose that $X \not= F$. As $X\setminus F$ is open, $X$ is quasiregular, and the fact that ${\cal B}$ is a $\pi$-base there exists $B\in{\cal B}$ such that $\overline{B}\subseteq X\setminus F$.
	Now, fix $x\in F$. We have $\overline{B}\cap\bigcap\{\overline{V}:V\in{\cal V}_x\}=\emptyset$. We will show that there exists $V\in{\cal V}_x$ such that $V\cap\overline{B}=\emptyset$. Suppose by way of contradiction that $V\cap\overline{B}\not=\emptyset$ for every $V\in{\cal V}_x$. The family ${\cal W}=\{V\cap\overline{B}:V\in{\cal V}_x\}$ is an open filter base on $\overline{B}$ as it is closed under finite intersections. ${\cal W}$ can then be extended to an open ultrafilter ${\cal U}$ on $\overline{B}$. As $\overline{B}$ is H-closed, ${\cal U}$ must converge to a point $p\in\overline{B}$. Therefore for every $V\in{\cal V}_x$ we have $p\in \overline{(V\cap\overline{B})}^{\overline{B}}=\overline{(V\cap\overline{B})}\cap\overline{B}\subseteq\overline{V}\cap\overline{B}$ and thus $p\in\overline{B}\cap\bigcap\{\overline{V}:V\in{\cal V}_x\}$. But this is a contradiction as $\overline{B}\cap\bigcap\{\overline{V}:V\in{\cal V}_x\}=\emptyset$. 
	
	Therefore for every $x\in F$ there exists $V_x\in{\cal V}_x$ such that $V_x\cap\overline{B}=\emptyset$.
	Clearly $\{V_x : x \in F\}$ is an open cover of $F$. Since $wL(X)$ is hereditary with respect to regular-closed sets, there
	exists ${\cal M} \in \{V_x : x \in F\}^{\leq \kappa}$ such that $F \subseteq \overline{\bigcup{\cal M}}$. Then there 
	exists $\alpha < \kappa^+$ such that ${\cal M} \in [\bigcup\{{\cal V}_x : x \in \overline{U_\alpha}\}]^{\leq \kappa}$. 
	As $\overline{B} \cap \bigcup {\cal M} = \emptyset$ it follows that $B\subset X\setminus\overline{\bigcup{\cal M}}$, hence 
	$X\setminus\overline{\bigcup{\cal M}}\not=\emptyset$. Thus, there exists $B_{\cal M}\in{\cal B}$ such that 
	$\emptyset\not= B_{\cal M} \subseteq U_{\alpha+1}\setminus\overline{\bigcup{\cal M}} \subseteq F\setminus\overline{\bigcup{\cal M}} = \emptyset$. 
	Since this is a contradiction, we conclude that $X = F$ and the proof is completed.
\end{proof}

We present an example that witnesses the fact that the hypotesis of quasiregularity is essencial. Such example is a Hausdorff not Urysohn not quasiregular space.
\begin{example}\rm\label{exe}
	A space $X$ having an H-closed $\pi$-base such that $|X|> 2^{wL(X)\chi(X)}$ (hence $|X|> 2^{wL(X)t(X)\psi_c(X)}$).
\end{example}
Let $\kappa$ be a cardinal. Consider the space $X=({\Bbb Q}\times \kappa )\cup ({\Bbb R}\setminus {\Bbb Q})$. If $(q,\alpha)\in {\Bbb Q}\times \kappa$, a basic open neighborhood is $U_n(q,\alpha)=\{(r,\alpha): r\in {\Bbb Q} \hbox{ and } |r-q|<\frac{1}{n} \}$ for each $n\in\omega$. If $x\in {\Bbb R}\setminus {\Bbb Q}$, a basic open neighborhood is $U_n(x)=\{x\}\cup\{(r,\alpha): r\in{\Bbb Q}, \alpha<\kappa \hbox{ and } |r-x|<\frac{1}{n} \}$. Clearly, $t(X)\psi_c(X)\leq \chi(X)\leq \omega$, moreover $wL(X)\leq {\frak c}$. The collection $\{U_n(q,\alpha): n\in\omega, q\in {\Bbb Q} \hbox{ and } \alpha<\kappa\}$ is a $\pi$-base in $X$. We prove that each $\overline{U_n(q,\alpha)}$ is H-closed  for each $n\in\omega, q\in {\Bbb Q} \hbox{ and } \alpha<\kappa$. Let $\cal U$ be a basic open cover of $\overline{U_n(q,\alpha)}$. For each $U\in {\cal U}$ there exists $W_U$ which is an open subset of $\Bbb R$ in the standard topology such that $\overline{U}=\{(q,\alpha): q\in \overline{W_U}\cap {\Bbb Q}\}\cup (({\Bbb R}\setminus {\Bbb Q})\cap \overline{W_U})$. Since $[q-\frac{1}{n}, q+\frac{1}{n}]$ is a compact (hence H-closed) subset of $\Bbb R$ and $\{W_U: U\in{\cal U}\}$ is an open cover of this interval, there exists a finite subfamily $\{U_1,...,U_k\}$ such that $[q-\frac{1}{n}, q+\frac{1}{n}]\subseteq \bigcup_{i=1}^{k}\overline{W_{U_i}}$. Therefore $\overline{U_n(q,\alpha)}\subseteq \bigcup_{i=1}^{k}\overline{U_i}$. Considering a cardinal $\kappa>2^{\frak c}$ we have that $|X|> 2^{wL(X)\chi(X)}$ (hence $|X|> 2^{wL(X)t(X)\psi_c(X)}$).$\hfill\triangle$

\bigskip

Notice that for the space $X$ in Example \ref{exe} $d_\theta(X)=\omega<d(X)=2^\omega$. Observe also that $d(X)\leq 2^{wL(X)\chi(X)}$ cannot hold in spaces with an H-closed $\pi$-base. Indeed, the inequality $|X|\leq d(X)^{\chi(X)}$ for Hausdorff spaces, will lead to the inequality $|X|\leq 2^{wL(X)\chi(X)}$ for spaces with an H-closed $\pi$-base, which is not true.\\

\bigskip

We can prove the following result.

\begin{lemma}\rm\label{rimcpt}
	Let $X$ be a space with a $\pi$-base $\cal B$ whose elements have compact boundaries and $F \subseteq X$ a closed set such that $X\backslash F \ne \varnothing$. Then there is $B \in {\cal B}$ such that $ \overline{B} \cap F = \varnothing$.
\end{lemma}
\begin{proof} There is $B \in {\cal B}$ such that $B \cap F = \varnothing$.  We are done if $\overline{B} \cap F = \varnothing$. Otherwise, $(\overline B\backslash B)\cap F \ne \varnothing$ and $L = (\overline B\backslash B) \cap F$ is compact.  Let $y \in B$.  As $X$ is Hausdorff, there are open sets $V_y$ and $W_y$ such that $y \in V$, $(\overline B\backslash B)\cap F \subseteq W_y$, and $V_y \cap W_y = \varnothing$.  There is $C \in {\cal B}$ such that $\varnothing \ne C \subseteq B \cap V_y$ implying $\overline{C} \subseteq \overline{B \cap V_y} \subseteq \overline{V_y}$.
	As  $\overline{C} \cap F \subseteq \overline{B} \cap F = (\overline{B}\backslash B) \cap F \subseteq W_y$ and $W \cap \overline{V_y} = \varnothing$,  it follows that  $\overline{C} \cap F = \varnothing$.
\end{proof}

By Lemma \ref{rimcpt} and by Theorem \ref{HclosedQR} we obtain the following results.

\begin{corollary}\rm
	A space with $\pi$-base whose elements have compact boundaries is quasiregular.
\end{corollary}

\begin{corollary}\rm
	If a space $X$ has an H-closed $\pi$-base whose elements have compact boundaries, then $|X|\leq 2^{wL(X)t(X)\psi_c(X)}$.
\end{corollary} 
	
\bigskip

Now we deal with the class of Urysohn spaces. 
We start with a well known result that characterizes compactness in the semigularization and we give the proof for sake of completeness. 

\begin{theorem}\rm \label{Katetov} (Kat\v etov)
	A space $X$ is Urysohn and H-closed if, and only if, $X_s$ is compact.
\end{theorem} 
\begin{proof} 
	Since a space is compact iff it is H-closed and regular it suffices to show that $X_s$ is regular. Let A be a regular-closed subspace of X and $p$ be a point not in $A$. For each $q\in A$, there are disjoint regular-open sets $U_q$ and $V_q$ such that $p\in U_q$, $q \in V_q$, and $\overline{U_q}$ and $ \overline{V_q}$ are disjoint.  There is a finite subset $F$ of $A$ such that $A$ is contained in $V = \bigcup \{\overline{V_q}:q \in F\}$. Let $U = \bigcap_{q\in F}U_q$.  Then $U \cap V = \emptyset$. We have that $V$ is a regular closed set and that $p\in X\setminus \overline{V}\subseteq X\setminus A$. This completes the proof.
\end{proof}
Using the previous theorem, we prove next result.
\begin{lemma}\rm \label{pi-base}
	Let $X$ be a Urysohn space. If there exists an H-closed $\pi$-base in $X$, then there exists a compact $\pi$-base in $X_s$. 
\end{lemma}
\begin{proof}
	Let ${\cal B}$ be an H-closed $\pi$-base in $X$. Since $\overline{B}=\overline{int(\overline{B})}=cl_s(int(\overline{B}))$, where $cl_s(A)$ denotes the closure in the semiregularization $X_s$, by Theorem \ref{Katetov}, we have that $cl_s(int(\overline{B}))$ is compact in $X_s$. We have to prove that $\{int(\overline{B}): B\in {\cal B}\}$ is a $\pi$-base in $X_s$. Fix a basic open set $int(\overline{V})$. There exists $B\in {\cal B}$ such that $B\subseteq V$. Then $int(\overline{B})\subseteq int(\overline{V})$. This completes the proof.
\end{proof}

In \cite{AK} Alas and Kocinac introduced for Hausdorff spaces, the cardinal function $k(X)$ that is the least cardinal $\kappa$ such that for every $x \in X$ there exists a family ${\cal V}_x$ of open neighborhoods of $x$ such that $|{\cal V}_x|\leq \kappa$ and for every regular closed subset $\overline{U}$, containing $x$, there exists $V\in {\cal V}_x$ such that $\overline{V}\subseteq \overline{U}$. It is straightforward that $k(X)\leq \chi(X)$. The proof of the following lemma is once again direct.

\begin{lemma}\rm\cite{AK}\label{kchi}
	$k(X)=\chi(X_s)$ for every space $X$.
\end{lemma}

\begin{theorem}\rm\label{k}
	Let $X$ be a Urysohn space with an H-closed $\pi$-base. Then $|X|\leq 2^{wL(X)k(X)}$.
\end{theorem}
\begin{proof}
	From Lemma \ref{pi-base} it follows that $X_s$ has a compact $\pi$-base. Since $wL(X_s)=wL(X)$ and by Lemma \ref{kchi}, we have that $|X|=|X_s|\leq 2^{wL(X_s)\chi(X_s)}= 2^{wL(X)k(X)}$.
\end{proof}

The following result is assumed without proof in \cite{DP} however a proof of it can be found in \cite{BCG}.
\begin{lemma}\rm\cite{DP}\label{p1}
	Let $X$ be an H-closed space. Then $\chi(X_s)\leq\psi_c(X)$.
\end{lemma}
Since the relations $\psi_c(X)= \psi_c(X_s)\leq\chi(X_s)$ are true for every space $X$, then we can say that $k(X)=\chi(X_s)=\psi_c(X)$ holds for H-closed spaces.
\begin{question}\rm
	Is it true that for Urysohn spaces having an H-closed $\pi$-base $k(X)=\psi_c(X)$?
\end{question}

The following example shows that Theorem \ref{k} is an actual improvement of the bound $|X|\leq2^{wL(X)\chi(X)}$ for Urysohn spaces having an H-closed $\pi$-base (in \cite{BC}).  

\begin{example}\rm
	There exist an Urysohn H-closed space $Z$ such that $k(Z)<\chi(Z)$. 
\end{example}
Consider the space $Z=2^\omega$ with the following topology on it: a basic open neighborhood of a point $\alpha$ is of the form $(U\setminus T)$, where $U$ is a basic open subset in $2^\omega$ containing $\alpha$ and $T$ is a subset of $2^\omega$ such that $|T|\leq \omega$ and $\alpha\not\in T$.
Clearly, $Z$ is Urysohn. We want to show that $Z$ is H-closed. 
Let $\cal U$ be an open cover of $Z$ made by basic open sets of the form $B\setminus S$. Since $2^\omega$ is compact there exist $B_1,...,B_k$ such that $Z=2^\omega \subseteq \bigcup_{i=1}^m B_i\subseteq \bigcup_{i=1}^m\overline{B_i}^{2^\omega}= \bigcup_{i=1}^m\overline{B_i\setminus S_i}^Z$. This proves that $Z$ is H-closed. Then every $\pi$-base is an H-closed $\pi$-base. As an additional remark, we can say that $wL(Z)<\omega$. It is easy to see that $(Z)_s=2^\omega$. Then, by Lemma \ref{kchi}, $k(Z)=\chi(2^\omega)=\omega$. It is straightforward to see that $\chi(Z)=|[2^\omega]^{\leq \omega}|=2^\omega$. $\hfill\triangle$\\

\smallskip
It can be easily proved that $\psi_\theta(X)\leq k(X)$ for every Urysohn space $X$. Then it is worthwhile to pose the following question.
\begin{question}\rm
	Does the inequality $|X|\leq 2^{wL(X)t_\theta(X)\psi_\theta(X)}$ hold for every Urysohn space $X$ with an H-closed $\pi$-base?
\end{question}

\section{Martin's Axiom and spaces with a quasiregular H-closed $\pi$-base} \label{2}

Recall the topological definition of Martin's Axiom for a cardinal $\kappa$,  $\aleph_0\leq\kappa<2^{\aleph_0}$. $MA(\kappa)$ states that for each compact space $X$ with the ccc, if $(U_\alpha:\alpha<\kappa)$ is a family of open dense subsets of $X$, then $\bigcap\{U_\alpha:\alpha<\kappa\}\not=\emptyset$.

\begin{theorem}\rm\cite{PW}\label{MAPW}
	Let $\aleph_0\leq\kappa<2^{\aleph_0}$. The followings are equivalent.
	\begin{enumerate}
		\item $MA(\kappa)$.
		
		\item\label{dense} For each space $X$ with the ccc such that $\{x\in X: x \hbox{ has a compact}$\\ $\hbox{ neighborhood}\}$ is dense in $X$, if $(U_\alpha:\alpha<\kappa)$ is a family of open dense subsets of $X$, then $\bigcap\{U_\alpha:\alpha<\kappa\}$ is a dense subset of $X$.
		
		\item \label{MAposet} For each poset $({\Bbb P}, \leq, 1)$ with the ccc, if $\cal D$ is a family of dense subsets of $\Bbb P$ with cardinality $<\kappa$ then there exists a filter $G\subseteq \Bbb P$ such that for every $D\in \cal D$ one has $D\cap G\not=\emptyset$.
	
	\end{enumerate}
\end{theorem}
\begin{theorem}\rm\label{MAqHclo}
	Let $\aleph_0\leq\kappa<2^{\aleph_0}$. The following is equivalent to the conditions in Theorem \ref{MAPW}.
	\begin{itemize}
		\item[4.]\label{quasiregHclos} For each quasiregular H-closed space $X$ with the ccc, if $(U_\alpha:\alpha<\kappa)$ is a family of open dense subsets of $X$, then $\bigcap\{U_\alpha:\alpha<\kappa\}\not=\emptyset$.
	\end{itemize}
\end{theorem}
\begin{proof}
	3.$\implies$4. Let $(X, \tau)$ be a quasiregular H-closed space having the ccc. Let $(U_\alpha:\alpha<\kappa)$ be a family of open dense subsets of $X$. Consider the following poset $({\Bbb P}, \leq,1)=(\tau\setminus\{\emptyset\},\subseteq, X)$. For each $\alpha <\kappa$ construct a family ${\cal E}_\alpha=\{U\in {\Bbb P}: \overline{U}\subseteq U_\alpha\}$. For the quasiregularity of $X$, the families ${\cal E}_\alpha $, $\alpha<\kappa$, are non-empty and since $U_\alpha$ is an open dense subset of $X$, ${\cal E}_\alpha$ is dense in $\Bbb P$ in the sense of posets. Then there exist a filter ${\cal F}\subseteq {\Bbb P}$, which is an open filter on $X$ such that ${\cal E}_\alpha\cap {\cal F}\not=\emptyset$ for every $\alpha <\kappa$. Then for every $\alpha <\kappa$ choose $F_\alpha \in {\cal E}_\alpha\cap {\cal F}$. Since $X$ is H-closed the adherence of $\cal F$ is non-empty, then there exists $x\in \bigcap\{\overline{F}: F\in{\cal F}\}$. Therefore $x\in \bigcap\{\overline{F_\alpha}: \alpha<\kappa\} \subseteq \bigcap\{U_\alpha: \alpha<\kappa\}$.\\
	The other implication is proved in the following theorem.	
\end{proof}

\begin{theorem}\rm
	Let $\aleph_0\leq\kappa<2^{\aleph_0}$. The followings are equivalent to the statements of Theorem \ref{MAPW} and Theorem \ref{MAqHclo}.
	\begin{enumerate}
		\item[5.]\label{denseqrH} For each space $X$ with the ccc such that $\{x\in X: x \hbox{ has a quasiregular}$\\ $\hbox{H-closed neighborhood}\}$ is dense in $X$, if $(U_\alpha:\alpha<\kappa)$ is a family of open dense subsets of $X$, then $\bigcap\{U_\alpha:\alpha<\kappa\}$ is a dense subset of $X$.
		
		\item[6.] \label{comp-pibase} For each space $X$ with the ccc and a compact $\pi$-base, if $(U_\alpha:\alpha<\kappa)$ is a family of open dense subsets of $X$, then $\bigcap\{U_\alpha:\alpha<\kappa\}$ is dense in $X$.
		
		\item[7.] \label{quasiregHclos-pibase} For each space $X$ with the ccc and a quasiregular H-closed $\pi$-base, if $(U_\alpha:\alpha<\kappa)$ is a family of open dense subsets of $X$, then $\bigcap\{U_\alpha:\alpha<\kappa\}$ is dense in $X$.
	\end{enumerate}
\end{theorem}
\begin{proof}
	4.$\implies$5. Let $X$ be a space with the ccc and let $D=\{x\in X: x$  has a quasiregular H-closed neighborhood$\}$ be a dense subset in $X$. Let $(U_\alpha: \alpha <\kappa)$ be a family of open dense subsets of $X$. Fix a non-empty open subset $W$ of $X$, then there exist $x\in W\cap D$ and a open neighborhood $V$ of $x$ such that $\overline{V}$ is quasiregular and H-closed. Therefore $Y=\overline{V\cap W}$ is a quasiregular H-closed space with the ccc. Moreover $U_\alpha\cap(W\cap V)$, for each $\alpha<\kappa$ is a non-empty open dense subset of $Y$. Then by hypotesis, $\emptyset\not=\bigcap\{U_\alpha \cap (W\cap V): \alpha <\kappa\}\subseteq \bigcap\{U_\alpha: \alpha <\kappa\} \cap W$.\\
	5.$\implies$7. Notice that if there exists a quasiregular H-closed $\pi$-base $\cal B$ then $\bigcup {\cal B}$ is a dense subset of $X$ and every point in $\bigcup{\cal B}$ has a quasiregular H-closed neighborhood.\\
	7.$\implies$6. is trivial.\\                          
	6.$\implies$2. Notice that if there exists a compact $\pi$-base $\cal B$ then $\bigcup {\cal B}$ is a dense subset of $X$ and every point in $\bigcup{\cal B}$ has a compact neighborhood.
\end{proof}

As a corollary one can obtain the following.

\begin{corollary}\rm
	If $X$ is a space with the ccc and a quasiregular H-closed $\pi$-base, then $X$ is Baire.
\end{corollary}

Actually we can remove the hypotesis ccc and obtain the following.

\begin{theorem}\rm\label{Baire}
	Let $X$ be a space with a quasiregular H-closed $\pi$-base. Then $X$ is a Baire space.
\end{theorem}

\begin{proof}
	Let $\{U_n:n<\omega\}$ be a family of open dense subsets of $X$ and ${\cal B}$ be a quasiregular H-closed $\pi$-base. We show $\bigcap_{n<\omega}U_n$ is dense in $X$. Let $U$ be a nonempty open set. We show that $U\cap\bigcap_{n<\omega}U_n\not=\emptyset$.
	
	As $U\cap U_0\not=\emptyset$, there exists $B\in{\cal B}$ such that $B\subseteq U\cap U_1$ and $\overline{B}$ is a quasiregular and H-closed subspace. Then there exists an open subset $B_0$ of $\overline{B}$ such that $\overline{B_0}^{\overline{B}}\subseteq B\subseteq U\cap U_0$. The set $B_0\cap(U\cap U_1)$ is an open subset of $\overline{B}$, then, by its quasiregularity, there exists an open subset $B_1$ of $\overline{B}$, such that $\overline{B_1}^{\overline{B}}\subseteq B_0\cap (U\cap U_1)$. Consider the open subset $B_1\cap (U\cap U_2)$ of $\overline{B}$ and continue by induction. Then we obtain a decreasing sequence $\{\overline{B_n}^{\overline{B}}: n<\omega\}$ of closed subsets of $\overline{B}$ such that $\overline{B_n}^{\overline{B}}\subseteq U\cap U_n$. Since $\overline{B}$ is H-closed we have that $\emptyset\not=\bigcap_{n<\omega}\overline{B_n}^{\overline{B}}\subseteq (\bigcap_{n<\omega}U_n)\cap U$ . 
\end{proof}

By Corollary \ref{equi}, we obtain the following corollary.

\begin{corollary}\rm
	Let $X$ be a quasiregular space with an H-closed $\pi$-base. Then $X$ is Baire.
\end{corollary}

\bigskip

In \cite{BBR} the authors showed that if $X$ is a Hausdorff Baire space with a rank 2-diagonal, then $|X|\leq wL(X)^
\omega$ and if $X$ is a Hausdorff Baire space with a $G_\delta$-diagonal, then $d(X) \leq wL(X)^\omega$. Therefore, as corollaries of these results and
Theorem \ref{Baire} we have the following:
\begin{corollary}\rm
	Let $X$ be a space with a quasiregular H-closed $\pi$-base.
	\begin{itemize}
		\item[(a)] If $X$ has a $G_\delta$-diagonal, then $d(X) \leq wL(X)^\omega$.
		\item[(b)] If $X$ has a rank 2-diagonal, then $|X|\leq wL(X)^\omega$.
	\end{itemize}	
\end{corollary}

{\bf Acknowledgement:} The author express his sincere gratitude to N. Carlson and M. Bonanzinga for their useful suggestions and inspiring discussions. He would also like to thank the \lq\lq National Group for Algebric and Geometric Structures, and their Applications\rq\rq (GNSAGA-INdAM) for their invaluable support throughout the course of this research.

\end{document}